\documentclass{article}
\usepackage{geometry}
\usepackage{latexsym,amsthm,amsmath,amssymb,url}
\usepackage[ruled, linesnumbered]{algorithm2e}
\usepackage[utf8]{inputenc}
\usepackage{tikz,authblk}
\usetikzlibrary{arrows.meta}
\usepackage{comment}
\usetikzlibrary{decorations.pathreplacing}

\newcommand{\YC}[1]{{#1}}
\newcommand{\YCb}[1]{{#1}}
\newcommand{\LH}[1]{{ #1}}
\newcommand{\GG}[1]{{ #1}}

\newtheorem{theorem}{Theorem}[section]
\newtheorem{lemma}[theorem]{Lemma}
\newtheorem{corollary}[theorem]{Corollary}

\newtheorem{conjecture}[theorem]{Conjecture}

\newtheorem{Problem}[theorem]{Problem}
\newtheorem{observation}[theorem]{Observation}

\title{Number of Subgraphs and Their Converses in Tournaments and New Digraph Polynomials}
\author{Jiangdong Ai\footnote{School of Mathematical Sciences and LPMC, Nankai University, Tianjin 300071, China. Partially supported
		by the Fundamental Research Funds for the Central Universities, Nankai University (No. 63243055). {\tt jd@nankai.edu.cn}.}, Gregory Gutin \thanks{Department of Computer Science. Royal Holloway University of London. {\tt g.gutin@rhul.ac.uk}.}, Hui Lei\footnote{School of Statistics and Data Science, LPMC and KLMDASR, Nankai University, Tianjin 300071, 
		China.  Partially supported
		by the NSFC grant (No.\ 12371351). {\tt hlei@nankai.edu.cn}.}, Anders Yeo \thanks {Department of Mathematics and Computer Science, University of Southern Denmark. {\tt andersyeo@gmail.com}.}, Yacong Zhou\thanks{Department of Computer Science. Royal Holloway University of London. {\tt Yacong.Zhou.2021@live.rhul.ac.uk}.} }
\date{}

\begin{document}
	
	\maketitle
	\begin{abstract}
		
		An oriented graph $D$ is {\it converse invariant} if, for any tournament $T$,  the number of copies of  $D$ in $T$ is equal to that of its converse $-D$. El Sahili and Ghazo Hanna [J. Graph Theory 102 (2023), 684-701] showed that any oriented graph $D$ with maximum degree at most 2 is converse invariant. They proposed a question: Can we characterize all converse invariant oriented graphs?
		
		In this paper, we introduce a digraph polynomial and employ it to give a necessary condition for an oriented graph to be converse invariant. This polynomial serves as a cornerstone in proving all the results presented in this paper. In particular, we characterize all orientations of trees with diameter at most 3 that are converse invariant. We also show that all orientations of regular graphs are not converse invariant if $D$ and $-D$ have different degree sequences. In addition, in contrast to the findings of \GG{El Sahili and Ghazo Hanna}, we prove that every \YC{connected} graph $G$ with maximum degree at least $3$, admits an orientation $D$ of $G$ such that $D$ is not converse invariant. We pose one conjecture.
	\end{abstract}	
	
	\section{Introduction}
	We refer to \cite{B-G} for terminology and notation not introduced here.   All graphs considered throughout this paper contain neither loops nor multiple edges.  An {\em arc} in a digraph from vertex $u$ to vertex $v$ is denoted by $(u,v)$. An {\it orientation} of a graph $G$ is a digraph $D$ obtained from $G$ by replacing every edge $uv$ by either arc $(u,v)$ or arc $(v,u)$. A digraph with at most one arc between any pair of vertices is called an {\em oriented graph} (or {\em orgraph} for short). A {\it tournament} is an orientation of a complete graph. 
	\GG{A {\em Hamiltonian dipath} in a
		tournament $T$} \LH{is a directed path containing all} \GG{vertices of $T$}. \LH{The definitions are similar for cycles.}
	Let $D=(V(D),A(D))$ be a digraph. \YC{For a vertex $v\in V (D)$, we denote by $d^+_D(v)$ and $d^-_D(v)$ the out-degree and in-degree of $v$, respectively. The {\it degree} of $v$ in $D$, denoted by $d_D(v)$, is the sum of its out-degree and in-degree in $D$. We may drop the $D$ subscripts if $D$ is clear from the context. We use $\Delta(H)$ to denote the maximum degree of the (di)graph $H$.}

	\LH{Counting subdigraphs in a tournament is a well-known problem.  The lower and upper bounds for the maximum number $P(n)$ of Hamiltonian dipaths in a tournament on $n$ vertices were intensively studied. Szele \cite{S1943} showed that $\frac{n!}{2^{n-1}} \leq P(n)\le O(\frac{n!}{2^{3n/4}})$.  Alon \cite{A1990} improved the upper bound by showing that $P(n)\le  O(n^{3/2}\frac{n!}{2^{n-1}})$. Adler, Alon and Ross \cite{AAR2001} improved the lower bound to $(e-o(1))\frac{n!}{2^{n-1}}$} \GG{and asked whether $P(n)=\Theta(\frac{n!}{2^{n-1}})$.} \LH{Fridgut and Kahn \cite{FK2005} improved the upper bound for $P(n)$ to} \GG{$O(n^{3/2-\xi}\frac{n!}{2^{n-1}})$, where $\xi\in [0.2507,0.2508]$}.  \LH{Moon \cite{M1972} and Busch  \cite{B2006} examined the upper and lower bounds for the minimum number of  Hamiltonian dipaths in strong tournaments on $n$ vertices. In 2007,  Moon and Yang \cite{M-Y}, constructed specific tournaments, referred to as `special chains', which determined the minimum number of Hamiltonian dipaths in strong tournaments on $n$ vertices. 
		Relevant results about Hamiltonian dicycles can be found in  \cite{AAR2001, A1990, C2007,FK2005,F-L,F-K-S,T1980}.
		We refer the readers to \cite{Y2017} for the problem of counting in tournaments for general spanning subdigraphs.}

	\LH{In this paper, we \GG{compare} the number of subdigraphs and their converses in every tournament $T$.} The {\em converse} of $D$, denoted by $-D$, is the digraph on the same vertex set where $(u, v)\in A(-D)$ if and only if $(v, u)\in A(D)$. Let $f_T(D)$ denote the number of copies of  $D$ in a tournament $T$. An orgraph $D$ is {\it converse invariant} if $f_T(D)=f_T(-D)$ for every tournament $T$. 
	
	Rosenfeld \cite{R1974} in 1974 proved that antidirected paths are converse invariant. El Sahili and Ghazo Hanna \cite{SH2023}  generalized Rosenfeld's result for any type of oriented paths, and also for cycles. They showed the following result.
	
	\begin{theorem}[\cite{SH2023}]\label{thm:SH}
		\YC{Let $G$ be a graph with $\Delta(G)\leq 2$. Then, every orientation $D$ of $G$ is converse invariant.}
	\end{theorem}
	
	In the same paper, El Sahili and Ghazo Hanna proposed the following natural problem.
	
	\begin{Problem}[\cite{SH2023}]
		Can we characterize all converse invariant orgraphs?   
	\end{Problem}

	In this paper, we introduce the following digraph polynomial of $x$
	\[P_D(x)=\sum_{u\in V(D)} (1+x)^{d^+_D(u)}(1-x)^{d^-_D(u)},\] 
	and use it to give the following necessary condition for an orgraph to be converse invariant. 
	\begin{theorem}\label{lem: polyn}
		Let $D$ be a converse invariant orgraph. Then $P_D(x)=P_{-D}(x)$.
	\end{theorem}
	
	Theorem \ref{lem: polyn} is utilized as a foundational tool to prove all results in this paper. In particular, we use Theorem \ref{lem: polyn} to show that paths and cycles are the only \YC{connected} graphs with the property that, every orientation of them is converse invariant. 
	
	\begin{theorem}\label{Thm:1}
		Let $G$ be a \YC{connected} graph with $\Delta(G)\geq 3$. Then, \YC{there exists an orientation $D$ of $G$ such that $D$ is not converse invariant}.
	\end{theorem} 
	
	In \cite{SH2023}, El Sahili and Ghazo Hanna showed that 
	directed $K_{1,d}$ (with all arcs going from the vertex with maximum degree) with $d\geq 3$ is not converse invariant by constructing a special tournament. Building upon their work, our paper establishes, using our method, that all orientations $D$ of stars $K_{1,d}$ with $d\geq 3$ are not converse invariant \YC{unless $D$ is isomorphic to $-D$ (denoted by $D\cong -D$)}. Furthermore, we extend this by showing a stronger result.
	
	\begin{theorem}\label{tree3}
		Let $D$ be an orientation of a tree with diameter at most three that is not a path. Then $D$ is converse invariant if and only if $D\cong -D$, or \YC{$D$ or $-D$ is isomorphic to the digraph shown in Figure \ref{fig:2}}. 
	\end{theorem}
	
	This paper is organized as follows. In Section \ref{sec:2}, we introduce a way to construct infinitely many converse invariant orgraphs that are not isomorphic to their converse. These constructions will be used in the characterization in Section \ref{sec:3}. In Section 3, we begin by proving Theorem \ref{lem: polyn}. Then using Theorem \ref{lem: polyn}, we show that all orientations of regular graphs are not converse invariant if $D$ and $-D$ have different degree sequences. At the end of this section, we prove Theorem \ref{Thm:1}. In Section \ref{sec:4},  we apply our method to characterize all orientations of trees with diameter at most 3 that are converse invariant. An
	open problem is proposed in the subsequent section.\\

	\noindent {\bf Additional Notation.} Given a digraph $D=(V(D),A(D))$. We use $|D|$ to denote the number of vertices and $\|D\|$ the number of arcs. We say that $v$ is a {\it source} if
	$d_D^-(v)=0$ and a {\it sink} if $d_D^+(v)=0$. The {\it degree sequence} of $D$, denote by $\YC{\mathrm{Deg}}(D)$, is $\{(d^+_D(v),d^-_D(v)): v\in V(D)\}$. Let $\mathrm{Deg}^+(D)=\{d^+_D(v): v\in V(D)\}$ and $\mathrm{Deg}^-(D)=\{d^-_D(v): v\in V(D)\}$. The digraph obtained by adding an arc $(u,v)$  to $D$ is denoted by $D+(u,v)$. We say that a pair of vertices $u$ and $v$ is {\it vertex-transitive} if there is an automorphism $\phi$ of $D$ such that $\phi(u)=v$ and $\phi(v)=u$. For any pair of digraphs $D$ and $D'$, we use $\mathrm{ism}(D, D')$ to denote the number of isomorphisms from $D$ to subdigraphs of $D'$, and $\YC{\mathrm{ism}}_{u\to v}(D, D')$ to denote the number of such isomorphisms given that $u\in V(D)$ is mapped to $v\in V(D')$. Let $\YC{\mathrm{aut}}(D)$ denote the number of automorphisms of $D$. A tournament $T$ is {\it transitive} if $(u,v)\in A(T)$ and $(v,w)\in A(T)$ imply that $(u,w)\in A(T)$. \YC{$P=x_1\ldots x_n$ is an {\it $(x_1,x_n)$-dipath} in $D$   if $x_ix_{i+1}\in A(D)$ for all $1\leq i< n$. A digraph $D$ is {\it strong} if there exists a $(u,v)$-dipath for every ordered pair $(u,v)$ of vertices in $D$.} Given a star $K_{1,d}$ with the central vertex $v$, we use $\overrightarrow{K}^i_{1,d}$ to denote the orientation of $K_{1,d}$ with $d^+(v)=i$. Let $G$ be a graph. The {\it distance} between two vertices is defined as the length of the shortest path connecting them. The {\it diameter} of $G$ is the greatest distance between any pair of vertices in $G$.
	
	\section{Recursive constructions for converse invariant orgraphs} \label{sec:2}
	
	The most natural condition for an orgraph $D$ to be converse invariant might be $D\cong-D$. In this section, we give a way to construct infinitely many converse invariant orgraphs $D$ with $D\not\cong -D$.
	
	The following observation is based on the fact that if $u$ and $v$ are vertex-transitive then $D+(u, v)\cong D+(v, u)$ and therefore there is a one-to-one correspondence between copies of $D$ and copies of $D+(u,v)$ in any tournament.
	\begin{observation}
		Let $D$ be an orgraph, and let $u$ and $v$ be a pair of non-adjacent vertices that are vertex-transitive. Then, $f_T(D)=f_T(D+(u,v))$ in any tournament $T$ (See Fig. \ref{fig:1} for an example).
	\end{observation}
	
	\begin{figure}[htbp]
		\centering
		\begin{minipage}{0.45\textwidth}
			\centering
			\begin{tikzpicture}[scale=1.5]
				\tikzstyle{vertexS}=[circle,draw, minimum size=8pt, scale=0.5, inner sep=0.5pt]
				\node (x) at (0,0) [vertexS]{};
				\node (u) at (-1,-1) [vertexS]{};
				\node (v) at (1,-1) [vertexS]{};
				\node (u') at (-2,-2) [vertexS]{};
				\node (v') at (2,-2) [vertexS]{};
				\node [above] at (-1.1,-1) {$u$};
				\node [above] at (1.1,-1) {$v$};
				\draw[thick, arrows={-Stealth[reversed, reversed]}] (u) to (v);
				\draw[thick, arrows={-Stealth[reversed, reversed]}] (x) to (u);
				\draw[thick, arrows={-Stealth[reversed, reversed]}] (u') to (u);
				\draw[thick, arrows={-Stealth[reversed, reversed]}] (x) to (v);
				\draw[thick, arrows={-Stealth[reversed, reversed]}] (v') to (v);
			\end{tikzpicture}
			\centering
			\caption{$D+(u,v)$}
			\label{fig:1}
		\end{minipage}
		\hfill
		\begin{minipage}{0.45\textwidth}
			\centering
			\begin{tikzpicture}[scale=1.5]
				\tikzstyle{vertexS}=[circle,draw, minimum size=8pt, scale=0.5, inner sep=0.5pt]
				\node (x) at (-1,0) [vertexS]{};
				\node (y) at (1,0) [vertexS]{};
				
				\node (u) at (-1,-1) [vertexS]{};
				\node (v) at (1,-1) [vertexS]{};
				\node (u') at (-1,-2) [vertexS]{};
				\node (v') at (1,-2) [vertexS]{};
				\node [left] at (-1.1,-1) {$u$};
				\node [right] at (1.1,-1) {$u$};
				\draw[thick, arrows={-Stealth[reversed, reversed]}] (u) to (v);
				\draw[thick, arrows={-Stealth[reversed, reversed]}] (x) to (u);
				\draw[thick, arrows={-Stealth[reversed, reversed]}] (u') to (u);
				\draw[thick, arrows={-Stealth[reversed, reversed]}] (y) to (v);
				\draw[thick, arrows={-Stealth[reversed, reversed]}] (v') to (v);
			\end{tikzpicture}
			\caption{$2D^+_u$}
			\label{fig:2}
		\end{minipage}
	\end{figure}

	The above observation provides a method to construct a new converse invariant orgraph from the existing ones by recursively adding arcs between vertex-transitive vertices. Therefore, we have the following observation.
	
	\begin{observation}\label{obs:2}
		Let $D$ be a converse invariant orgraph, and let $u$ and $v$ be a pair of non-adjacent vertices that are vertex-transitive. Then, $D+(u,v)$ is also a converse invariant orgraph. 
	\end{observation}
	
	\YC{The following operation was introduced in \cite{ZZ2020}}. For any digraph $D$ and $u\in V(D)$, we can obtain a new digraph $2D^+_{u}$ by taking two copies of $D$ and then adding an arc between the two vertices corresponding to $u$ (see Fig \ref{fig:2} for an instance). And we say that $2D^+_{u}$ is obtained by {\em bridge-mirroring} $D$ at vertex $u$. \YC{In \cite{ZZ2020}, Zhao and Zhou showed} \GG{that every orgraph which has the same density in all tournaments of sufficiently large order} \YC{can only be obtained by starting from an arc and recursively applying bridge-mirroring operation. However, by Obervation \ref{obs:2} we can construct converse invariant orgraphs by applying this operation recursively started from any converse invariant orgraphs.} The following results can be obtained directly by applying the above observations. Note that if $D$ is converse invariant, then the union of arbitrarily many disjoint copies of $D$ is also converse invariant.

	
	\begin{lemma}\label{obs:4}
		Let $D$ be a converse invariant orgraph and $D'$ an orgraph obtained by bridge-mirroring $D$ at a vertex in $D$. Then, $D'$ is also a converse invariant orgraph.
	\end{lemma}
	\begin{proof}
		Let $2D$ (resp. $2(-D)$) denote the  union of two disjoint copies of $D$ (resp. $-D$). Since $D$ is converse invariant, then we have the following holds for any tournament $T$, 
		
		$$f_T(D')=f_T(2D)=\sum_{T_1,T_2} f_{T_1}(D)\cdot f_{T_2}(D)=\sum_{T_1,T_2} f_{T_1}(-D)\cdot f_{T_2}(-D)=f_T(2(-D))=f_T(-D'),$$ 
		where $T_1$ and $T_2$ are taken over all pairs of disjoint subtournaments of $T$ with order $|D|$. Hence, $D'$ is a converse invariant orgraph.
	\end{proof}
	It is natural to ask whether there is an integer $k$ such that every orgraph $D$ with $D\not\cong -D$ and maximum degree at least $k$  is not  converse invariant. The following simple corollary of Lemma \ref{obs:4} shows that such $k$ does not exist.
	\begin{corollary}
		Let $k\geq 3$ be any integer, then there is an orgraph $D$ with maximum degree $k$ such that $D\not\cong -D$ and $f_T(D)=f_T(-D)$. 
	\end{corollary}

	\section{A Polynomial of Degree Sequences}\label{sec:3}
	
	We start this section from proving Theorem \ref{lem: polyn}. 	
	
	\vspace{2mm}	
	\noindent{\bf Theorem}$\mbox{ }${\bf \ref{lem: polyn}.}
	{\em		Let $D$ be a converse invariant orgraph. Then $P_D(x)=P_{-D}(x)$.}
	\begin{proof}
		Let $T$ be a random tournament on $|D|-1$ vertices where each direction of the arc between any pair of vertices is taken independently and uniformly with probability $1/2$. Let $p$ be an arbitrary real in the interval $[-1/2, 1/2]$. Obtain a new random tournament from $T$ by adding a new vertex $v^*$, and for any vertex $v\in V(T)$, we add the arc $v^*v$ with probability $1/2+p$ and $vv^*$ with probability $1/2-p$. Denote the resulting random tournament by $T_p$.  Thus,
		
		\begin{eqnarray}\label{Ism}
			\mathbb{E}(\mathrm{ism}(D,T_p))&=&\sum_{u\in V(D)} \mathbb{E} (\mathrm{ism}_{u\to v^*}(D,T_p))\notag\\
			&=& \sum_{u\in V(D)}  (|D|-1)! 2^{-\|D\|+d^-(u)+d^+(u)}(1/2+p)^{d^+(u)}(1/2-p)^{d^-(u)}\notag\\
			&=&  (|D|-1)! 2^{-\|D\|}\sum_{u\in V(D)} (1+2p)^{d^+(u)}(1-2p)^{d^-(u)}.
		\end{eqnarray}
		
		Observe that $\mathrm{aut}(D)=\mathrm{aut}(-D)$. Thus, for any $p\in[-1/2,1/2]$ and $T^*\in T_p$, we have $\mathrm{ism}(D,T^*)= \mathrm{ism}(-D,T^*)$ because $f_{T^*}(D)=f_{T^*}(-D)$ and $f_{T^*}(D)=\frac{\mathrm{ism}(D,T^*)}{\mathrm{aut}(D)}$. Therefore, $\mathbb{E}(\mathrm{ism}(D,T_p))= \mathbb{E}(\mathrm{ism}(-D,T_p))$ for any $p\in[-1/2,1/2]$.  
		
		Then by (\ref{Ism}), for any $p\in [-1/2, 1/2]$ we have
		\begin{eqnarray}\label{infi}		
			\sum_{u\in V(D)} (1+2p)^{d^+(u)}(1-2p)^{d^-(u)}=\sum_{u\in V(D)} (1+2p)^{d^-(u)}(1-2p)^{d^+(u)}.
		\end{eqnarray}
		As (\ref{infi}) holds for infinitely many $p$, $P_D(x)=P_{-D}(x)$.\qed
	\end{proof}
	
	The following corollary can be obtained from  Theorem \ref{lem: polyn} by taking $x=1$.
	\begin{corollary}\label{cor:sources and sinks}
		Let $D$ be an orgraph, $S_1$ the set of its sources, and $S_2$ the set of its sinks. If $D$ is converse invariant, then \[\sum_{v\in S_1}2^{d^+(v)}=\sum_{v\in S_2}2^{d^-(v)}.\]
	\end{corollary}
	A polynomial  
	$f(x)$ is  {\it even}  if  $f(x)=f(-x)$ for all $x$ in the domain of $f$. Let $c_t(D)$ denote the coefficient  of term $x^t$ in $P_D(x)$. Then 
	
	\[c_t(D)=\sum\limits_{v\in V(D)}\sum\limits_{i=0}^{t} {d^+(v) \choose i} {d^-(v)\choose t-i} (-1)^{t-i}.\]
	
	Let $\Delta_o(D)$ the maximum odd number such that $\Delta_o(D)\leq \Delta(D)$. We have the following result. 
	\begin{theorem}\label{thm:1}
		Let $D$ be an orgraph. Then $P_D(x)=P_{-D}(x)$ if and only if $c_t(D)=0$ for any odd number $t\leq \Delta_o(D)$.
	\end{theorem}
	\begin{proof}
		Observe that $P_D(x)=P_{-D}(x)$ if and only if the polynomial $P_D(x)$ is even. Since $P_D(x)$ is even if and only if the  coefficient $c_t(D)$ of term $x^t$ is zero for any odd number $t\leq \Delta_o(D)$, we are done.
	\end{proof}
	
	Note that $c_1(D)=0$ holds for all digraphs since $\sum_{v\in V(D)} d^+(v)=\sum_{v\in V(D)} d^-(v)$. And $c_3(D)=0$ if and only if 
	\begin{equation}\label{cod:c3}
		\sum_{v\in V(D)}(d^+(v)-d^-(v))^3=3\sum_{v\in V(D)} (d^+(v)^2-d^-(v)^2).
	\end{equation}

	By Theorem \ref{lem: polyn} and Theorem \ref{thm:1}, we have the following corollary. 
	
	\begin{corollary}
		Let $D$ be a converse invariant orgraph. If $\Delta(D)$ is even, then $\Delta_o(D)=\Delta(D)-1$, and
	\YCb{	\begin{equation}\label{cod:cd}
				c_{\Delta_o(D)}(D)=\sum_{v: d(v)=\Delta(D)} ((-1)^{d^-(v)}d^+(v)+(-1)^{d^-(v)-1}d^-(v))+\sum_{v: d(v)=\Delta_o(D)} (-1)^{d^-(v)}=0.
		\end{equation}}
		If $\Delta(D)$ is odd, then $\Delta_o(D)=\Delta(D)$, and
		\YCb{\begin{equation}\label{cod:cd2} 
				c_{\Delta_o(D)}(D)=\sum_{v:d(v)=\Delta_o(D)} (-1)^{d^-(v)}=0.
		\end{equation}}
	\end{corollary}
	
	By (\ref{cod:cd2}), the following holds.
	
	\begin{corollary}\label{cor3.6}
		Let $G$ be a graph with $\Delta(G)\geq 3$. If $\Delta(G)$ is odd and $G$ \YCb{has an odd number of vertices} with the maximum degree, then for any oriented $D$ of $G$, there exists a tournament $T$ such that $f_T(D)\neq f_T(-D)$. 
	\end{corollary}

	We can also easily show the following result about the orientations of regular graphs by using Theorem \ref{lem: polyn}.
	
	\begin{theorem}
		Let $d$ be a positive integer and $D$ an orientation of a $d$-regular graph. If $\mathrm{Deg}(D)\neq \mathrm{Deg}(-D)$, then there exists a tournament $T$ such that $f_T(D)\neq f_T(-D)$.
	\end{theorem}
	
	\begin{proof}
		Suppose to the contrary that $f_T(D)=f_T(-D)$ for every tournament $T$. Then by Theorem \ref{lem: polyn}, when $x\neq 1$ we have
		\[\sum_{u\in V(D)} \left(\frac{1+x}{1-x}\right)^{d^+(u)}=\frac{P_D(x)}{(1-x)^{d(u)}}=\frac{P_{-D}(x)}{(1-x)^{d(u)}}=\sum_{u\in V(D)}\left(\frac{1+x}{1-x}\right)^{d^-(u)},\]
		which implies the two polynomials $\sum_{u\in V(D)} y^{d^+(u)}$ and $\sum_{u\in V(D)} y^{d^-(u)}$ are the same. Hence, $\mathrm{Deg}^+(D)=\mathrm{Deg}^-(D)$. Let $a^D_{i,d-i}$ denote the number of vertices with out-degree $i$ and in-degree $d-i$ in $D$. Note that $a^D_{i,d-i}=|\{v: d^+(v)=i\}|$ and $a^{-D}_{i,d-i}=a^D_{d-i,i}=|\{v: d^-(v)=i\}|$. As $\mathrm{Deg}^+(D)=\mathrm{Deg}^-(D)$,  we have that $a^D_{i,d-i}=a^{-D}_{i,d-i}$ for all $i\in [d]$ and so $\mathrm{Deg}(D)=\mathrm{Deg}(-D)$, a contradiction. 
	\end{proof}
	
	To end this section, we now give the proof of Theorem \ref{Thm:1}.
	
	\noindent{\it Proof of Theorem \ref{Thm:1}.}
	We first assume that $G$ is a tree. Then we orient it from any leaf $u$ such that apart from $u$, every other vertex has in-degree exactly one.  The
	resulting orgraph is denoted by $D$. Then, we have one source $u$ with out-degree one and at least two sinks with in-degree one in $D$ since $\Delta(G)\geq 3$. Let $S_1$ be the set of sources in $D$ and $S_2$ the set of its sinks. As 
	\[\sum_{v\in S_1}2^{d^+(v)}=2<2+2\leq \sum_{v\in S_2}2^{d^-(v)},\]
	we are done by applying Corollary \ref{cor:sources and sinks}.
	
	We now assume that $G$ contains cycles. Let $C$ be one of the shortest cycles in $G$. Assume without loss of generality that $V(G)=\{v_1,v_2,\dots, v_n\}$ and $C=v_1v_2\dots v_gv_1$. Note that $g<n$ since $\Delta(G)\geq 3$ and $C$ is a shortest cycle. We obtain an orientation $D$ of $G$ by setting $(v_g, v_1)\in A(D)$ and $(v_i, v_j)\in A(D)$ if $i<j$ and $v_iv_j\in E(G)\setminus\{v_1v_g\}$. Let $T$ be a tournament obtained from a strong tournament $T_0$ with  $|T_0|=g$ and a transitive tournament $T_1$ with  $|T_1|=n-g$ by adding all arcs from $T_0$ to $T_1$. One can observe that $f_T(D)>0$ but $f_T(-D)=0$, which completes the proof.\qed

	\section{Trees with diameter at most three }\label{sec:4}
	
	In this section, we characterize all orientations of trees with diameter at most three that are converse invariant.  In \cite{SH2023}, Sahili and Hanna showed that $\overrightarrow{K}^d_{1,d}$ with $d\geq 3$
	is not converse invariant by constructing a special tournament. The following corollary generalizes this result avoiding an explicit construction.
	
	\begin{corollary}\label{star}
		Let $D$ be an orientation of star $K_{1,d}$ with $d\geq 3$. Then $D$ is converse invariant if and only if $D\cong -D$. 
	\end{corollary}
	\begin{proof}
		Let $v$ be  the central vertex of $K_{1,d}$. If $D$ is converse invariant, then by Corollary \ref{cor:sources and sinks},  $d^+(v)=d^-(v)$ and therefore $D\cong -D$ when $v$ is not a sink or source. And $v$ cannot be a source or sink since otherwise, by Corollary \ref{cor:sources and sinks}, we have $2^d=2d$ which contradicts the fact that $d\geq 3$. 
	\end{proof}
	
	We now characterize all orientations of trees with diameter $3$ that are converse invariant.
	We also call a tree with diameter $3$ a {\em double star} and the {\em central vertices} of a double star are the vertices with degree at least $2$.
	\begin{theorem}\label{doublestar}
		Let $D$ be an orientation of a double star that is not a path. Then $D$ is converse invariant  if and only if $D\cong -D$ or $D$ is formed by bridge-mirroring $\overrightarrow{K}^0_{1,2}$  or $\overrightarrow{K}^2_{1,2}$ at the central vertex.
	\end{theorem}
	\begin{proof}
		Let $u$ and $v$ be the central vertices of $D$. Since $D$ is not an orientation of a path, we have $d(u)\geq 2$,  $d(v)\geq 2$ and $\max\{d(u),d(v)\}\geq 3$. Without loss of generality, we assume $(u, v) \in A(D)$. Thus, $u$ cannot be a sink and $v$ cannot be a source. 
		Suppose that $D$ is converse invariant. We consider the following three cases.
		
		\vspace{2mm}		
		
		{\bf Case 1. Exactly one of $u$ and $v$ is a sink or source.}
		
		\YCb{	Note that there are two subcases: (i) $u$ is a source and $v$ is not a sink, and (ii) $u$ is not a source and $v$ is a sink.
			Now we consider Subcase (i). }
		By Corollary \ref{cor:sources and sinks}, we have
		$2^{d^+(u)}+2(d^-(v)-1)=2(d^+(u)-1)+2d^+(v)$ and so
		\begin{equation}\label{eq:t1}
			2^{d^+(u)}-2d^+(u)=2(d^+(v)-d^-(v)).
		\end{equation}
		Suppose  that $d^+(u)=2$.  Then $d^+(v)=d^-(v)$ by (\ref{eq:t1}). Since $\max\{d(u),d(v)\}\geq 3$, we have $d(v)\geq 4$. Assume that $d^+(v)=d^-(v)=n-1$. Then, $|V(D)|=2n$. Note that $n\geq3$ as $d(v)\geq 4$. Let $T$ be a tournament with $V(T)=\{v_1,v_2,\ldots,v_{2n}\}$ and $A(T)=(\{(v_i, v_j): 1\leq i<j\leq 2n\}\setminus \{(v_1, v_n)\})\cup \{(v_n, v_1)\}$. Now we consider $f_T(D)$ and $f_T(-D)$. Note that in $T$, only the out-degree and in-degree of $v_{n+1}$ are both at least $n-1$. Then the vertex $v$ of $D$ can only be mapped to the vertex $v_{n+1}$ of $T$. A simple calculation shows that $f_T(D)=\frac{n(n-1)}{2}$ and $f_T(-D)=n(n-1)$. Therefore, $f_T(D)\neq f_T(-D)$, a contraction.

		Suppose that  $d^+(u)\geq 3$. 
		Thus, if $d^-(v)\geq 2$, then by (\ref{eq:t1}) we have
		\YCb{\begin{eqnarray}\label{vu+}
				d^+(v)+d^-(v)-1\geq d^+(v)+1\geq 2^{d^+(u)-1}-d^+(u)+3\geq d^+(u)+1, 
		\end{eqnarray}}
		which implies \YCb{$d(v)=d^+(v)+d^-(v)\geq d^+(u)+2=d(u)+2$}. Then by (\ref{cod:cd2}), \YCb{$d(v)$} is even. And now, by (\ref{cod:cd}) \YCb{and the fact that $d(v)\geq d(u)+2$ (and therefore $\{v:d(v)=\Delta_o(D)\}=\emptyset$)}, we have $d^+(v)=d^-(v)$ which contradicts  (\ref{eq:t1}) and the fact that $d^+(u)\geq 3$. Note that if $d^-(v)=1$ and $d^+(v)\geq d^+(u)$, then (\ref{vu+}) also holds and we can use the same argument.
		
		It remains to consider the case when $d^-(v)=1$ and $d^+(v)\leq d^+(u)$. By (\ref{eq:t1}), we have   $d^+(v)=2^{d^+(u)-1}-d^+(u)+1$. Since $d^+(v)\leq d^+(u)$, we have $d^+(u)=3$ and therefore $d^+(v)=2$. This implies that $D$ is obtained by bridge-mirroring $\overrightarrow{K}^2_{1,2}$ at the central vertex. \GG{This completes Subcase (i).}
		
		\YCb{Observe that Subcase (ii) is Subcase (i) for $-D$ instead of $D$. Thus, in Subcase (ii) we can show that $D$ is obtained by bridge-mirroring $\overrightarrow{K}^0_{1,2}$ at the central vertex. }
		
		\vspace{2mm}	
		
		{\bf Case 2. $u$ is a source and $v$ is a sink.}
		
		By Corollary \ref{cor:sources and sinks}, we have $2^{d^+(u)}+2d^-(v)-2=2^{d^-(v)}+2d^+(u)-2$ and therefore
		\[2^{d^+(u)}-2d^+(u)=2^{d^-(v)}-2d^-(v).\]
		This implies $d^+(u)=d^-(v)$ and thus $D\cong -D$.
		
		\vspace{2mm}	
		
		{\bf Case 3. $u$ is not a source and $v$ is not a sink.}
		
		By Corollary \ref{cor:sources and sinks}, we have that the number of sinks equal to the number of sources which means 
		\begin{equation}\label{eq:t3}
			d^+(u)+d^+(v)=d^-(u)+d^-(v).
		\end{equation}
		
		Suppose that $d^+(v)=d^-(u)$. Then $d^+(u)=d^-(v)$ by (\ref{eq:t3}) and therefore $D\cong -D$. Suppose that $d^+(v)\neq d^-(u)$. Assume that $d^+(v)>d^-(u)$ (the case $d^+(v)<d^-(u)$ can be proved similarly). Since $u$ is not a source, $d^-(u)\geq 1$. Let $T$ be a transitive tournament with $V(T)=\{v_1,v_2,\ldots,v_{|D|}\}$ and $A(T)=\{(v_i, v_j): 1\leq i<j\leq |D|\}$. Let $T'$ be a tournament with $V(T')=V(T)$ and $A(T')=A(T) \setminus \{(v_1, v_{d^-(u)+1})\})\cup \{(v_{d^-(u)+1}, v_1)\}$. Note that $T'\not\cong T$ because $d^-(u)+1\geq 2$.
		We first consider $f_{T'}(D)$. 
		Since $(u, v) \in A(D)$, $u$ is mapped to a vertex in $T'$ with a subscript less than the subscript of the vertex to which $v$ is mapped in $T'$ for any copy of $D$ in $T'$. Let $a$ denote the number of copies of $D$ in $T$ where the vertex $u$ in $D$ is mapped to the vertex $v_{d^-(u)+1}$ in $T$. Note that $a\geq 1$. Then $f_{T'}(D)=f_T(D)-a$. Since $T\cong -T$, we have $f_T(D)=f_T(-D)$ and thus $f_{T'}(D)=f_T(-D)-a$. Now we consider $f_{T'}(-D)$. As $d^+(v)>d^-(u)$, in any copy of $-D$ in $T'$, $v$ must have been mapped to a vertex with a higher subscript than $d^-(u)+1$ and therefore $f_{T'}(-D)=f_T(-D)$. Then, $f_{T'}(D)=f_T(-D)-a=f_{T'}(-D)-a$. Hence,  $f_{T'}(D)\neq f_{T'}(-D)$, a contradiction. 		
	\end{proof}
	
	Corollary \ref{star} and Theorem \ref{doublestar} imply Theorem \ref{tree3}.

	\section{Conclusion}
	
	In this paper, we introduce a new digraph polynomial and use it to show a necessary condition for an orgraph to be converse invariant. In particular, we use this necessary condition to characterize converse invariant orientations of trees with diameters bounded by three. We would like to make the following conjecture.
	
	\begin{conjecture}\label{ctrees}
		Let $D$ be an orientation of a tree with maximum degree at least $3$. Then $D$ is converse invariant if and only if $D\cong -D$ or $D$ can be obtained by the bridge-mirroring operation recursively from an orientation of a path. 
	\end{conjecture}
	
	An {\em out-branching} $B^+_s$ (respectively, {\em in-branching}
	$B^-_s$) is an orientation of a tree in which each vertex $x\neq s$ has precisely one arc entering (leaving) it and $s$ has no arcs entering (leaving) it. The vertex $s$ is the root of $B^+_s$ (respectively, $B^-_s$).  Conjecture \ref{ctrees} might be difficult to verify. As a special case,  one can consider out-branchings or in-branchings for Conjecture \ref{ctrees}.

\end{document}